\documentclass[10pt]{amsart}
\usepackage[english]{babel}
\usepackage{amssymb}
\usepackage{amsmath}
\usepackage{srcltx}
\usepackage{lscape}

\newtheorem{dummy}{Dummy}

\newtheorem{lemma}[dummy]{Lemma}
\newtheorem{theorem}[dummy]{Theorem}

\theoremstyle{definition}
\newtheorem{definition}{Definition}

\newtheorem{example}[dummy]{Example}

\newtheorem{remark}[dummy]{Remark}

\newcommand{\tl}{\widetilde{\tau}}

\newcommand{\ignore}[1]{}

\author{S. Pumpl\"un}
\email{susanne.pumpluen@nottingham.ac.uk}
\address{School of Mathematical Sciences\\
University of Nottingham\\
University Park\\
Nottingham NG7 2RD\\
United Kingdom
}

\keywords{Space-time block code, linear code, nonassociative algebra, coset code, wiretap code, skew polynomial ring}

\subjclass[2010]{Primary: 17A35; Secondary: 11T71, 94B40, 94B05}

\begin{document}

\title[Quotients of orders]
{Quotients of orders in algebras obtained from skew polynomials with applications to coding theory}

\maketitle

\begin{abstract}
 We describe families of nonassociative finite unital rings that occur as quotients of natural nonassociative orders in
generalized nonassociative cyclic division algebras over number fields.
These natural orders have already been  used to systematically construct fully diverse
fast-decodable space-time block codes. We show how the quotients of natural orders can be
 employed for coset coding.
 Previous results by Oggier and Sethuraman involving quotients of  orders in associative cyclic division algebras
  are obtained as special cases.
\end{abstract}

%
%

\section*{Introduction}

Let $S$ be a unital ring, $\sigma$ an injective endomorphism of $S$ and $\delta$ a left $\sigma$-derivation of $S$.
Take a monic skew polynomial $f\in R=S[t;\sigma,\delta]$ of degree $m$ greater than one.
Then the additive subgroup
$\{h\in S[t;\sigma,\delta]\,|\, {\rm deg}(h)<m \}$ of $R=S[t;\sigma,\delta]$ becomes a nonassociative unital  ring
via the multiplication $g\circ h=gh \,\,{\rm mod}_r f $, using right division by $f$  \cite{P66}, \cite{P15}.

This nonassociative ring is denoted by $S_f=R/Rf$ and is an algebra
over the subring $S_0=\{a\in S\,|\, ah=ha \text{ for all } h\in S_f\}$ of $S$. If $f$ is
an invariant skew polynomial,  meaning $Rf$ is a two-sided ideal, we obtain the usual associative
 quotient ring $R/Rf$.  We call $S_f$ a \emph{Petit algebra}, as the construction goes back to Petit \cite{P66}.

Adapting the approach from \cite{DO} and \cite{OS}, we  work with monic skew polynomials $f$ with coefficients either
 in the rings of integers $\mathcal{O}_K$ of
 a number field $K$, or in a natural $\mathcal{O}_F$-order
 of a cyclic division algebra over a number field $F$. We  define a natural nonassociative
order $\Lambda$ in $S_f$ and investigate the  nonassociative quotient rings
 of $\Lambda$ by a two-sided ideal $\mathcal{I}$ in $\Lambda$.
These quotient rings are isomorphic to the direct sum of
 Petit algebras over a finite ring. We thus generate a large class of finite nonassociative rings
 which can be viewed as quotient rings of  natural orders $\Lambda$,
 and which includes the  class of associative rings described in \cite{OS} as a special case.
 We put a strong emphasis on algebras obtained for $\delta=0$ and $f=t^m-d$, since
 these algebras are behind the design of recent fast-decodable
fully diverse  space-time block codes (cf. \cite{SP14}, \cite{R13}, \cite{MO13}, see \cite{PS15}).
Moreover, due to the connection between the algebras $S_f$ and cyclic $(f,\sigma,\delta)$-codes  \cite{BL}, these particular algebras
define $\sigma$-cyclic codes if $f$ is reducible.

 The finite nonassociative rings we obtain can be employed
for the coset encoding of space-time block codes, analogously as described in \cite[Section 8]{OS} for the associative case, but also
for the coset encoding of linear $(f,\sigma,\delta)$-codes as described in \cite[Section 5.2, 5.3]{DO}
and \cite{P16.1}.

In this paper, we will focus on the coset encoding of space-time block codes. Space-time block codes are
used for reliable high rate transmission
over wireless digital channels with multiple antennas transmitting and receiving the data.
A \emph{space-time block code} (STBC) is a set $\mathcal{C}$ of complex $n\times m$ matrices
 that satisfies a number of properties which determine how well the code performs.
We consider a model representing slow multiple antenna fading channels, which means that the channel is constant over $nL$ channel
uses, so the code contains $n\times nL$ codewords of the type $X=(X_1,\dots,X_L)$, with the $X_i$
matrices in some $\mathcal{C}$.

We proceed as follows:
In Section \ref{sec:1} we collect the terminology and results needed later
and explain how to get coset codes using
quotients of natural orders in generalized nonassociative cyclic division algebras. In Sections
\ref{sec:naturalI} and \ref{sec:naturalII} we define natural orders in certain $S_f$ and look at their quotients.
We recall how fully diverse space-time codes are obtained from an order in a nonassociative division algebra $S_f$
over a number field and then give examples of coset codes.

The different structures of the quotients of a natural order are then investigated in Section \ref{sec:structure}.
We put particular emphasis on generalized nonassociative  cyclic division algebras and their natural orders, because of their role
in designing both linear $\sigma$-constacyclic codes and in building fast-decodable space-time block codes.
 We prove that we can restrict ourselves to the case when the two-sided ideal that is factored out has the form
 $\mathfrak{q}^s\Lambda$, where $\mathfrak{q}$ is a prime ideal in a suitable subring of integers,
and then again limit our investigation to generalized nonassociative cyclic algebras.
The lower bound  for the determinant of a sum of positive-definite
matrices given in (\ref{equ:key}) needed for coding gain estimates and established in \cite{OS} holds analogously
in our setting.
Sections \ref{sec:inertial} and \ref{sec:inertialII} look at different choices for the prime ideal
$\mathfrak{q}$, with Section \ref{sec:inertial}
focusing  on the case when $\mathfrak{q}$ is an inertial ideal.

We do not strive for completeness and refrain from investigating all possible cases
 of nonassociative finite rings which can be obtained as quotients of natural orders.
 It is clear how to proceed after seeing the selected cases highlighted here.
Potential future applications to coding theory are briefly considered in Section \ref{sec:last}.

%
%

\section{Preliminaries}\label{sec:1}

\subsection{Nonassociative algebras}
Let $R$ be a unital commutative ring and let $A$ be an
$R$-module.
We call $A$ an \emph{algebra} over $R$ if there exists an
$R$-bilinear map $A\times A\to A$, $(x,y) \mapsto x \cdot y$, denoted simply by juxtaposition $xy$,
the  \emph{multiplication} of $A$.
An algebra $A$ is called \emph{unital} if there is
an element in $A$, denoted by 1, such that $1x=x1=x$ for all $x\in A$.
We will only consider unital algebras.

For an $R$-algebra $A$, associativity in $A$ is measured by the {\it associator} $[x, y, z] = (xy) z - x (yz)$.
The {\it left nucleus} of $A$ is defined as ${\rm Nuc}_l(A) = \{ x \in A \, \vert \, [x, A, A]  = 0 \}$, the
{\it middle nucleus}  as ${\rm Nuc}_m(A) = \{ x \in A \, \vert \, [A, x, A]  = 0 \}$ and  the
{\it right nucleus}  as ${\rm Nuc}_r(A) = \{ x \in A \, \vert \, [A,A, x]  = 0 \}$.
Their intersection
 ${\rm Nuc}(A) = \{ x \in A \, \vert \, [x, A, A] = [A, x, A] = [A,A, x] = 0 \}$ is the {\it nucleus} of $A$.
${\rm Nuc}_l(A)$, ${\rm Nuc}_m(A)$ and ${\rm Nuc}_r(A)$ are associative
subalgebras of $A$ containing $R1$.
The  {\it commuter} of $A$ is defined as ${\rm Comm}(A)=\{x\in A\,|\,xy=yx \text{ for all }y\in A\}$ and
the {\it center} of $A$ is ${\rm C}(A)=\text{Nuc}(A)\cap  {\rm Comm}(A)$ \cite{Sch}.

Let $R$ be a Noetherian integral domain with quotient field $F$ and $A$ a finite-dimensional unital $F$-algebra.
Then an $R$-\emph{lattice} in $A$ is a finitely generated submodule $\Gamma$ of $A$ which contains an $F$-basis of $A$.
An $R$-\emph{order} $\Gamma$ in $A$ is a multiplicatively closed $R$-lattice containing $1_A$
(note that the multiplication need not be associative).
An $R$-order will be called \emph{maximal} if $\Gamma'\subset\Gamma$ implies $\Gamma'=\Gamma$ for every $R$-order
$\Gamma'$ in $A$.

   A non-trivial algebra $A$ over a field $F$ is called a {\it division algebra} if for any $a\in A$, $a\not=0$,
the left multiplication  with $a$, $L_a(x)=ax$,  and the right multiplication with $a$, $R_a(x)=xa$, are bijective maps.
Any division algebra is simple, that means has only trivial two-sided ideals.
A finite-dimensional algebra
$A$ is a division algebra over $F$ if and only if $A$ has no zero divisors.


\subsection{Skew polynomial rings}


Let $S$ be a unital associative ring, $\sigma$ an injective ring endomorphism of $S$ and
$\delta:S\rightarrow S$ a \emph{left $\sigma$-derivation}, i.e.
an additive map such that
$$\delta(ab)=\sigma(a)\delta(b)+\delta(a)b$$
for all $a,b\in S$, implying $\delta(1)=0$. The \emph{skew polynomial ring} $R=S[t;\sigma,\delta]$
is the set of skew polynomials $\sum_{i=0}^{n}a_it^i$
with $a_i\in S$, where addition is defined term-wise and multiplication by
$$ta=\sigma(a)t+\delta(a) \quad (a\in S).$$
Define ${\rm Fix}(\sigma)=\{a\in S\,|\,
\sigma(a)=a\}$ and ${\rm Const}(\delta)=\{a\in S\,|\, \delta(a)=0\}$ and
put $S[t;\sigma]=S[t;\sigma,0]$ and $S[t;\delta]=S[t;id,\delta]$.

 For $f=\sum_{i=0}^{n}a_it^i$ with $a_n\not=0$ define the \emph{degree} of $f$ as ${\rm deg}(f)=n$ and ${\rm deg}(0)=-\infty$.
Then ${\rm deg}(fg)\leq{\rm deg} (f)+{\rm deg}(g)$ (with equality if $f$ or $g$ has an invertible leading coefficient,
if $S$ is a domain  or if $S$ is a division ring).
 A skew polynomial $f\in R$ is \emph{irreducible} in $R$ if it is no unit and  it has no proper factors, i.e if there do not exist $g,h\in R$ with
 ${\rm deg}(g),{\rm deg} (h)<{\rm deg}(f)$ such
 that $f=gh$.  $f\in R$ is a \emph{(right)-invariant} (also called \emph{two-sided}) skew polynomial
if $fR\subset Rf$. In the following, we drop the right when we talk about invariant polynomials.
  If $f$ is invariant then $Rf$ is a two-sided ideal in $R$ and conversely, every two-sided ideal in $R$ is generated by
   an invariant polynomial.


\subsection{The algebras $S_f$}\label{sec:S_f}


 Let $R=S[t;\sigma,\delta]$, $\sigma$ be injective and $\delta$ a left $\sigma$-derivation.
 Suppose that $f=\sum_{i=0}^{m}d_it^i\in R$ has an invertible leading coefficient $d_m$.
Then for all $g\in R$  there exist  uniquely determined $r,q\in R$ with ${\rm deg}(r)<{\rm deg}(f)$, such that
$$g(t)=q(t)f(t)+r(t).$$
Let ${\rm mod}_r f$ denote the remainder of right division by such an $f$.
  Let  $R_m=\{g\in R\,|\, {\rm deg}(g)<m\}.$ Then the additive group $R_m$  together with the multiplication
$$g\circ h= gh \,\,{\rm mod}_r f $$
becomes a  unital nonassociative ring $S_f=(R_m,\circ)$ also denoted by $R/Rf$ \cite{P15}.
$S_f$ is a unital nonassociative algebra  over $S_0=\{a\in S\,|\, ah=ha \text{ for all } h\in S_f\}$.
This construction was  introduced by Petit \cite{P66} for unital division rings $S$.
We call $S_f$ a \emph{Petit algebra}.
 $S_f$ is associative if and only if  $Rf$ is a two-sided ideal in $R$.
We will only consider monic $f$, since monic $f$ are the ones used in code constructions.
Moreover, $S_f=S_{af}$ for all invertible $a\in S$.

  If $S_f$ is not associative then $S\subset {\rm Nuc}_l(S_f)$ and $S\subset{\rm Nuc}_m(S_f)$,
${\rm Nuc}_r(S_f)=\{g\in R\,|\, fg\in Rf\}$ and $S_0$ is the center of $S_f$.
It is easy to see that
$$C(S)\cap {\rm Fix}(\sigma)\cap {\rm Const}(\delta)\subset S_0.$$
 Right multiplication with $0\not=h\in S_f$,
$R_h:S_f\longrightarrow S_f,$ $p\mapsto ph$, is an $S$-module endomorphism \cite{P66}.
By expressing the map $R_h$ in matrix form
with respect to the $S$-basis $1,t,\dots, t^{m-1}$ of $S_f$, the map
$$\gamma: S_f \to {\rm End}_K(S_f), h\mapsto R_h$$
induces an injective $S$-linear map
$$\gamma: S_f \to {\rm Mat}_m(S), h\mapsto R_h \mapsto Y.$$
This special characteristic of $S_f$ is exploited when designing  space-time block codes.
It uses the fact that $S\subset {\rm Nuc}_l(S_f)$ and $S\subset{\rm Nuc}_m(S_f)$.

 If $S$ is a division algebra and $S_f$ is a finite-dimensional vector space over $S_0$,
  then $S_f$ is a division algebra if and only if $f$ is irreducible in $R$ \cite[(9)]{P66}.

For $f=\sum_{i=0}^{m}d_it^i\in S[t;\sigma]$, $t$ is left-invertible if and only if $d_0$ is invertible
 by a simple degree argument. Thus if $f$ is irreducible (hence $d_0\not=0$) and $S$ a division ring then
$t$ is always left-invertible and $S_0={\rm Fix}(\sigma)\cap C(S)$, which also is the center of $S_f$
\cite{P15}.

We highlight  two special cases that are particularly relevant for our coding applications later:

\begin{definition}
(i) Let $S/S_0$ be an extension of commutative unital rings and $\sigma$ an automorphism of $S$ of order $m$
such that $S_0\subset{\rm Fix}(\sigma)$. For any $c\in S$ (it may even be zero),
$$S_f=S[t;\sigma]/S[t;\sigma] (t^m-c)$$
is called a \emph{nonassociative cyclic algebra}  $(S/S_0,\sigma,c)$ \emph{of degree $m$}.
For $c\in S_0$, this is an associative cyclic algebra, cf. \cite{DO}, \cite{OS}.
For $c \in S \setminus S_0$,  $(S/S_0,\sigma,c)$ has nucleus containing
$S$  and center containing $S_0$. Over fields, these algebras were studied for instance in
\cite{S} or \cite{S12}.
\\ (ii) Let $D$ be a finite-dimensional central division algebra
 over  $F={\rm Cent}(D)$ of degree $n$ and $\sigma\in {\rm Aut}(D)$ such that
$\sigma|_{F}$ has finite order $m$.
A \emph{(generalized) nonassociative cyclic algebra of degree $m$}  is an algebra
$S_f=D[t;\sigma]/D[t;\sigma]f$ over $F_0=F\cap {\rm Fix}(\sigma)$
with $f=t^m-d\in D[t;\sigma]$.  We denote this algebra by $(D,\sigma, d)$.
\end{definition}

\begin{example}\label{ex:gencyclic}
Let $F$ and $L$ be fields, $F_0=F\cap L$, and let $K$ be a cyclic field extension of both $F$ and $L$ such that
${\rm Gal}(K/F) = \langle \rho\rangle$ and $[K:F] = n$, ${\rm Gal}(K/L) = \langle \sigma \rangle$ and $[K:L] = m$,
and such that $\rho$ and $\sigma$ commute.
 Let $D=(K/F, \rho, c)$  be an associative cyclic division algebra over $F$ of degree $n$,
  $c\in F_0$ and $d \in D^\times$.
 For $x= x_0 + x_1 e+x_2  e^2+\dots + x_{n-1}e^{n-1}\in D$, extend $\sigma$ to an automorphism $\sigma\in {\rm Aut}_L(D)$
of order $m$  via
$$\sigma(x)=\sigma(x_0) +  \sigma(x_1)e +\sigma(x_2)  e^2 +\dots +\sigma(x_{n-1}) e^{n-1}.$$
  For all $d \in D^\times$,
$$S_f=D[t;\sigma]/D[t;\sigma](t^m-d)$$
is the generalized nonassociative cyclic algebra $(D,\sigma,d)$ of dimension $m^2n^2$ over $F_0$.
The algebra is associative if and only if $d\in F_0$.
For all $d \in F^\times$,
$$(D, \sigma, d)\cong (L/F_0, \gamma, c)\otimes_{F_0} (F/F_0, \sigma, d),$$
i.e.  it is the tensor product of an associative and a nonassociative cyclic algebra \cite{P16}.
For $f\in F_0[t]$ this algebra appears in the classical literature on associative central simple algebras as
a \emph{generalized cyclic algebra}  of degree $m$  in \cite[Section 1.4]{J96}.

The algebras  $(D,\sigma,d)$ with $d\in L^\times$ or $d\in F^\times$
 are used to construct fast-decodable space-time block codes,  the matrix
representing their right multiplication with entries in $K$, i.e. computed with respect to the canonical basis
of $S_f$ as a left $K$-vector space, yields the codebooks in this case, cf. for instance \cite{P13.2}, \cite{PS15}, \cite{R13}.
\end{example}

\subsection{STBCs and coset coding}

A \emph{space-time block code} (STBC) is a set $\mathcal{C}$ of complex $n\times m$ matrices
 that satisfies a number of properties which determine how well the code performs.
 $\mathcal{C}$ is called \emph{fully diverse} if the difference of any two
code words  has full rank, so that $\det(X-X')\not=0$ for all matrices $X\not=X',$ $X,X'\in \mathcal{C} $.
 Since our codes $\mathcal{C}$ will be based on the matrix representing right multiplication in
 an algebra, they are linear and thus their
 \emph{minimum determinant} is given by
$$\delta(\mathcal{C})=\inf_{0\not=X\in \mathcal{C}}|\det(X)|^2.$$
If $\delta(\mathcal{C})$
 is bounded below by a constant, even if the codebook $\mathcal{C}$   is infinite,
 the code $\mathcal{C} $ has \emph{non-vanishing determinant} (NVD).
 If $\mathcal{C}$ is fully diverse,
$\delta(\mathcal{C})$ defines the \emph{coding gain} $\delta(\mathcal{C})^{\frac{1}{n}}$, and
the larger $\delta(\mathcal{C})$ is, the better the error performance of the code is expected to be.

We consider a model representing slow multiple antenna fading channels, which means that the channel is constant over $nL$ channel
uses, so the code contains $n\times nL$ codewords of the type $X=(X_1,\dots,X_L)$, with the $X_i$
matrices in some $\mathcal{C}$.

To construct a coset space-time block code we take the following approach: we take a
space-time block code $\mathcal{C}$ which corresponds to a natural order $\Lambda$ inside a generalized nonassociative
cyclic division algebra, this code is automatically fully diverse. A \emph{generalized nonassociative
cyclic division algebra} is a Petit algebra $S_f=K[t;\sigma]/K[t;\sigma](t^m-d)$ with $K$ a number field,
 or a Petit algebra $S_f=D[t;\sigma]/K[t;\sigma](t^m-d)$
with $D=(K/F,\rho,c)$ a cyclic division algebra of degree $n$ over a number field, each time
$f=t^m-d$ an irreducible polynomial chosen suitably as explained in the next section.
Then $\mathcal{C}$ consists of the square matrices with entries in $\mathcal{O}_K$
which represent the right multiplication in $S_f$ by a non-zero element, computed with respect to the canonical basis of
 $S_f$ as a left $\mathcal{O}_K$-module.
This is the \emph{inner code}.
The quotient of $\Lambda$ by a suitably chosen  two-sided ideal $\mathcal{J}$  generated
by an ideal $\mathcal{I}$ of $\mathcal{O}_F$ as in Section \ref{sec:naturalI} (or by an ideal $\mathcal{I}$ of a suitable subring of $\mathcal{O}_F$
in case we work with
$f\in D[t;\sigma]$ as in Section \ref{sec:naturalII})
 yields a finite nonassociative unital algebra $\Lambda/\mathcal{J}$, over which we design a
 code $\overline{\mathcal{C}}$ of length $L$.
This code $\overline{\mathcal{C}}$ is  the \emph{outer code}  and is an additive subgroup of
$$\bigoplus_{i=1}^{L}{\rm Mat}_s(\mathcal{O}_K/\mathcal{I}\mathcal{O}_K),$$
 with $s=m$ or $s=mn$.
Its codewords are of the type $X=(X_1,\dots,X_L)$, with the $X_i$
matrices representing the right multiplication in the algebra $\Lambda/\mathcal{J}$.

In order to obtain the $X_i$, we take the entries in the matrices representing the right multiplication in $\Lambda$, i.e. the entries of the
 matrices of  $\mathcal{C}$, and read them  modulo $\mathcal{I}\mathcal{O}_K$,
which gives the outer space-time block code $\overline{\mathcal{C}}$.
 The \emph{coset code} is obtained as the additive subgroup $\mathcal{C}'$ of matrices in
 $\bigoplus_{i=1}^{L}{\rm Mat}_s(\mathcal{O}_K)$
satisfying
$\pi(\mathcal{C}')=\overline{\mathcal{C}},$ where $\pi:\bigoplus_{i=1}^{L}{\rm Mat}_s(\mathcal{O}_K)\longrightarrow
 \bigoplus_{i=1}^{L}{\rm Mat}_s(\mathcal{O}_K/\mathcal{I}\mathcal{O}_K)$.
 By construction, $\mathcal{C}'$ has length $L$ and is contained in the inner code $\mathcal{C}$.
 The goal is to design a well performing code $\mathcal{C}'$, so that for instance it has a large minimum determinant, or
 is fast-decodable. The later is guaranteed automatically if the outer code is fast-decodable.

%
%

\section{Quotients of natural orders in $S_f$, I} \label{sec:naturalI}

When $S$ is a field, every skew polynomial ring $S[t;\sigma,\delta]$ can be made into either a twisted or a
differential polynomial ring by a linear change of variables. When constructing linear codes, however,
 it can be an advantage to consider general skew polynomial rings. For instance,
 cyclic $(f,\sigma,\delta)$-codes constructed from natural order algebras $S_f$ obtained from some monic
 $f\in S[t;\sigma,\delta]$ can
 produce better distance bounds than cyclic $(f,\sigma,\delta)$-codes constructed only with an automorphism,
 i.e. with $\delta=0$, as noted in \cite{BU14}. Therefore we keep a more  general skew polynomial setup in Sections \ref{subsec:assumptions}
and \ref{subsec:naturalI}, although
when applying the results to space-time codes later in this paper, we always assume that $\delta=0$.

\subsection{} \label{subsec:assumptions}
Let $K/F$ be a Galois extension of number fields of degree $n$. Let
$\mathcal{O}_F$ and $\mathcal{O}_K$ be the rings of integers of $F$ and $K$.
Let $\mathcal{I}$ be an ideal of $\mathcal{O}_F$ and
 $\pi:\mathcal{O}_K\longrightarrow \mathcal{O}_K/\mathcal{I}\mathcal{O}_K$ be
the canonical projection.
Let  $\sigma\in {\rm Gal}(K/F)$.
We have $\sigma(\mathcal{I}\mathcal{O}_K)\subset \mathcal{I}\mathcal{O}_K$ since
$\sigma|_F=id$. Thus $\sigma$ induces a ring homomorphism
$$\overline{\sigma}:\mathcal{O}_K/\mathcal{I}\mathcal{O}_K \longrightarrow \mathcal{O}_K/\mathcal{I}\mathcal{O}_K$$
with $\sigma=\overline{\sigma} \circ \pi$ and ${\rm Fix}(\overline{\sigma})=\mathcal{O}_F/\mathcal{I}$.

Suppose that $\delta$ is an $F$-linear left $\sigma$-derivation on $K$ such that
$\delta(\mathcal{O}_K)\subset \mathcal{O}_K$.
 Then $\delta$ induces a left $\overline{\sigma}$-derivation
 $$\overline{\delta}:\mathcal{O}_K/\mathcal{I}\mathcal{O}_K \longrightarrow \mathcal{O}_K/\mathcal{I}\mathcal{O}_K.$$
Since $\mathcal{O}_F$ is  a Dedekind domain,
$$\mathcal{I}=\mathfrak{q}_1^{s_1}\cdots\mathfrak{q}_t^{s_t}$$
for suitable prime ideals $\mathfrak{q}_i$ of $\mathcal{O}_F$ and so
$$\mathcal{O}_F/\mathcal{I}=\mathcal{O}_F/\mathfrak{q}_1^{s_1}\cdots\mathfrak{q}_t^{s_t}
\cong \mathcal{O}_F/\mathfrak{q}_1^{s_1}\times \cdots\times \mathcal{O}_F/\mathfrak{q}_t^{s_t},$$
$$\mathcal{O}_K/\mathcal{I}\mathcal{O}_K= \mathcal{O}_K/ \mathfrak{q}_1^{s_1} \cdots\mathfrak{q}_t^{s_t} \mathcal{O}_K\cong
\mathcal{O}_K/ \mathfrak{q}_t^{s_t}\mathcal{O}_K\times\dots\times \mathcal{O}_K/ \mathfrak{q}_t^{s_t} \mathcal{O}_K,$$
by the Chinese Remainder Theorem.
On each ring $\mathcal{O}_K/ \mathfrak{q}_i^{s_i}\mathcal{O}_K$ there is a canonical induced action of $\sigma$,
and a canonical derivation  induced by $\delta$.


\subsection{ }  \label{subsec:naturalI}


Suppose
$$f=\sum_{i=0}^{m}d_it^i\in \mathcal{O}_K[t;\sigma,\delta]$$
 is a monic skew polynomial,  irreducible in $K[t;\sigma,\delta]$.
 Consider the division algebra
$$S_f=K[t;\sigma,\delta]/K[t;\sigma,\delta] f$$
 over $S_0$.  Since $\sigma\in {\rm Gal}(K/F)$ and $\delta$ is $F$-linear, $S_f$ is an algebra over $S_0=F$.
The nonassociative  $\mathcal{O}_{F}$-algebra
$$\Lambda=\mathcal{O}_K[t;\sigma,\delta]/\mathcal{O}_K[t;\sigma,\delta] f$$
 is an $\mathcal{O}_{F}$-order in $S_f$ called the \emph{natural order}. $\Lambda$
 is uniquely determined whenever $f$ is not invariant, since then $K$ is the left nucleus of $S_f$ which uniquely
 determines $\mathcal{O}_{K}$ and in turn $\Lambda$.
 Since $f$ is irreducible in $K[t;\sigma,\delta]$, $\Lambda$ does not have any zero divisors.
Since $\mathcal{O}_{F}$ lies in the center of $\Lambda$  \cite{P16.1}, for every  ideal $\mathcal{I}$ in $\mathcal{O}_{F}$,
$\mathcal{I}\Lambda$ is a two-sided ideal of $\Lambda$. We have
\begin{equation}\label{equ:quotient}
\mathcal{I}\Lambda=\{al\,|\, a\in\mathcal{I},l\in\Lambda\}=\left\{\sum_{i=0}^{m-1}a_it^i\,|\, a_i\in
 \mathcal{I}\mathcal{O}_K \right\}.
 \end{equation}
 The surjective homomorphism of nonassociative rings
$$\Psi:\Lambda\longrightarrow  (\mathcal{O}_K/\mathcal{I}\mathcal{O}_K)[t;\overline{\sigma},\overline{\delta}]/
(\mathcal{O}_K/\mathcal{I}\mathcal{O}_K)[t;\overline{\sigma},\overline{\delta}]\overline{f}$$
$$g\mapsto \overline{g}$$
has kernel $\mathcal{I}\Lambda$ and
induces an isomorphism between the two unital nonassociative algebras given by
\begin{equation}\label{eq:relevant}
\Lambda/\mathcal{I}\Lambda\longrightarrow
(\mathcal{O}_K/\mathcal{I}\mathcal{O}_K)[t;\overline{\sigma},\overline{\delta}]/(\mathcal{O}_K/\mathcal{I}\mathcal{O}_K)
[t;\overline{\sigma},\overline{\delta}]\overline{f}=S_{\overline{f}},
\end{equation}
$$g+\mathcal{I}\Lambda \mapsto \overline{g}.$$
These are algebras over $\mathcal{O}_F/\mathcal{I}$. The algebra
 $S_{\overline{f}}$ is associative  if and only if $R\overline{f}$ is a two-sided ideal in
$R=(\mathcal{O}_K/\mathcal{I}\mathcal{O}_K)[t;\overline{\sigma},,\overline{\delta}]$.
The associative orders  from \cite{DO} appear here for invariant polynomials $f$ and $\delta=0$.

\subsection{} \label{ex:nca1}
Let $K/F$ be a Galois extension of degree $m$, ${\rm Gal}(K/F)=\langle\sigma\rangle$ and
 $A=S_f$ with $f=t^m-d\in \mathcal{O}_K[t;\sigma]$ irreducible in $K[t;\sigma]$.
Then $A=(K/F,\sigma,d)$ is a nonassociative cyclic division algebra of degree $m$ over $F$.
A natural order of $A$ is given by the $\mathcal{O}_F$-algebra $\Lambda=\mathcal{O}_K[t;\sigma]/\mathcal{O}_K[t;\sigma]f$,
and as a left $\mathcal{O}_K$-module,
$$\Lambda=\mathcal{O}_K \oplus  \mathcal{O}_K t\oplus \dots \oplus\mathcal{O}_K  t^{m-1}.$$
 The right multiplication in $\Lambda$ with $a=a_{m-1}t^{m-1}+\dots+a_1t+a_0$ is given by the
$m \times m$ matrix
\begin{equation} \label{equ:matrix_rep_cda}
\gamma(a) = \left[ \begin{array}{ccccc}
a_0 & d \sigma(a_{m-1})& d \sigma^2(a_{m-2}) & \cdots & d \sigma^{m-1}(a_1) \\
a_1 & \sigma(a_0) & d\sigma^2(a_{m-1}) & \cdots & d \sigma^{m-1}(a_{2}) \\
a_2 & \sigma(a_1) & \sigma^2(a_0) & \cdots & d \sigma^{m-1}(a_3)\\
\vdots & \vdots & \vdots & \ddots & \vdots \\
a_{m-1} & \sigma(a_{m-2}) & \sigma^2(a_{m-3}) & \cdots & \sigma^{m-1}(a_0) \end{array} \right].
\end{equation}
with entries in $\mathcal{O}_K$.
If $d\in \mathcal{O}_K\setminus \mathcal{O}_F$, $A$ is not associative and
 $\Lambda$ is uniquely determined.
 Since $\mathcal{O}_F$ lies in the center of $\Lambda$,
for any  ideal $\mathcal{I}$ of $\mathcal{O}_F$, $\mathcal{I}\Lambda$ is a two-sided ideal of
$\Lambda$.
We have the following $(\mathcal{O}_F/\mathcal{I}\mathcal{O}_F)$-algebra isomorphism:
\begin{equation}\label{eq:generalcyclic}
\Lambda/\mathcal{I}\Lambda\cong ((\mathcal{O}_K/\mathcal{I}\mathcal{O}_K)/(\mathcal{O}_F/\mathcal{I}\mathcal{O}_F),
\overline{\sigma},\bar{d})=S_{\overline{f}}
\end{equation}
with $\bar{d}=d+\mathcal{I}$, $\overline{\sigma}(a+\mathcal{I}\mathcal{O}_K)=
\sigma(a)+\mathcal{I}\mathcal{O}_K$ for all $a\in \mathcal{O}_K$, and
$$\overline{f}(t)=t^m-\bar{d}\in (\mathcal{O}_K/\mathcal{I}\mathcal{O}_K)[t;\overline{\sigma}].$$

If $d\in \mathcal{O}_F$ is non-zero, $A$ is an associative cyclic division algebra, $S_{\overline{f}}$
is an associative generalized cyclic algebra (and if $\bar{d}\not=0$, a classical associative cyclic algebra),
 and $\Lambda$ depends on the choice of the maximal subfield $K$ in $A$.  This case
  is covered in \cite{DO} and \cite{OS} and these associative algebras were employed in the code constructions in \cite{SPO12}.

Equation (\ref{equ:quotient}) and the isomorphism in (\ref{eq:generalcyclic}) mean
 that the right multiplication in $\Lambda/\mathcal{I}\Lambda$ is given by the
$m \times m$ matrix in (\ref{equ:matrix_rep_cda}) where the entries now are read modulo $\mathcal{I}\mathcal{O}_K$.
We call this matrix $\gamma(\overline{a})$. Therefore we can obtain coset codes  by taking the
pre-image of codewords $(\gamma(\overline{x_1}),\dots,\gamma(\overline{x_L}))$  under
$$ \pi:\bigoplus_{i=0}^L {\rm Mat}_{m}(\mathcal{O}_K) \longrightarrow \bigoplus_{i=0}^L {\rm Mat}_{m}
(\mathcal{O}_K/\mathcal{I}\mathcal{O}_K).$$

%
%

\section{Quotients of natural orders in $S_f$, II} \label{sec:naturalII}

\subsection{}\label{subsec:iteratedI}

Let $K/F$ be a cyclic Galois extension of number fields of degree $n$ with  ${\rm Gal}(K/F)=\langle \rho\rangle$.
 Let $\mathcal{O}_F$ and $\mathcal{O}_K$ be the corresponding rings of integers.
Let $D=(K/F, \rho, c)$ be a cyclic division algebra over $F$ such that $c\in \mathcal{O}_F^\times$.
 Let $\mathcal{D}=(\mathcal{O}_K/\mathcal{O}_F,\rho, c)$
be the generalized associative cyclic algebra over $\mathcal{O}_F$ of degree $n$ such that
 $\mathcal{D} \otimes_{\mathcal{O}_F}F=(K/F, \rho, c)=D$.
 Then
$$\mathcal{D}=\mathcal{O}_K \oplus   \mathcal{O}_Ke \oplus \dots  \oplus\mathcal{O}_K  e^{n-1}$$
is a natural $\mathcal{O}_F$-order of $D$, cf. \ref{subsec:naturalI} or \cite{DO}.

Let $\sigma\in {\rm Aut}(D)$ and $\delta$ be a $\sigma$-derivation on $D$, satisfying the following
criteria:
\begin{itemize}
\item $F_0=F\cap {\rm Fix}(\sigma)\cap {\rm Const}(\delta)$ is a number field.
\item $\sigma(\mathcal{D})\subset \mathcal{D}$ and $\delta(\mathcal{D})\subset \mathcal{D}$.
\item $S_0=\mathcal{O}_F\cap {\rm Fix}(\sigma)\cap {\rm Const}(\delta)$ is  the ring of integers of $F_0$ where here
$\sigma$ and $\delta$ denote the restrictions of $\sigma$ and $\delta$ to $\mathcal{D}$.
\end{itemize}

Suppose
$$f=\sum_{i=0}^{m}d_it^i\in \mathcal{D}[t;\sigma,\delta]$$
 is a monic skew polynomial, irreducible in $D[t;\sigma,\delta]$. Consider the division algebra
$$S_f=D[t;\sigma,\delta]/D[t;\sigma,\delta] f$$
 over $F_0$. Then the $S_0$-order
$$\Lambda=\mathcal{D}[t;\sigma,\delta]/\mathcal{D}[t;\sigma,\delta] f=
\mathcal{O}_K \oplus  \mathcal{O}_K e\oplus  \dots\oplus\mathcal{O}_K  e^{n-1}t^{m-1}$$
 is the \emph{natural order} of $S_f$.
The center of $\Lambda$ contains $S_0$.
Since $f$ is irreducible in $D[t;\sigma,\delta]$, $\Lambda$ does not have  zero divisors  \cite{P16.1}.

Let $\mathcal{I}$ be an ideal in $S_0$.  $S_0$
is contained in the center of $\mathcal{D}$ and the center of $\Lambda$, thus
 $\mathcal{I}\mathcal{D}$ is a two-sided ideal of $\mathcal{D}$
and $\mathcal{I}\Lambda$ is a two-sided ideal of $\Lambda$.  We have
\begin{equation}\label{equ:quotientII}
\mathcal{I}\Lambda=\{al\,|\, a\in\mathcal{I},l\in\Lambda\}=\left\{\sum_{i=0}^{m-1}a_it^i\,|\, a_i\in
 \mathcal{I}\mathcal{O}_K\right\}.
 \end{equation}
Let $\pi:\mathcal{D}\longrightarrow \mathcal{D}/\mathcal{I}\mathcal{D}$ be
the canonical projection. We have $\sigma( \mathcal{I}\mathcal{D})\subset \mathcal{I}\mathcal{D}$ since
 $\mathcal{I}\subset{\rm Fix}(\sigma)$ and $\sigma(\mathcal{D})\subset \mathcal{D}$ by  assumption.
 Therefore $\sigma$ induces a ring homomorphism
$$\overline{\sigma}:\mathcal{D}/\mathcal{I}\mathcal{D} \longrightarrow \mathcal{D}/\mathcal{I}\mathcal{D}$$
with
$${\rm Fix}(\overline{\sigma})={\rm Fix}(\sigma)/\mathcal{I}{\rm Fix}(\sigma)$$
and $\sigma=\overline{\sigma} \circ \pi$.
We also have
$\delta( \mathcal{I}\mathcal{D})\subset  \mathcal{I}\mathcal{D}$ by  assumption.
 That means $\delta$ induces a left $\overline{\sigma}$-derivation
 $$\overline{\delta}:\mathcal{D}/\mathcal{I}\mathcal{D} \longrightarrow \mathcal{D}/\mathcal{I}\mathcal{D}$$
 with field of constants
 $${\rm Const}(\overline{\delta})={\rm Const}(\delta)/\mathcal{I}.$$
 The surjective homomorphism of nonassociative rings
$$
 \Psi:\Lambda\longrightarrow  (\mathcal{D}/\mathcal{I}\mathcal{D})[t;\overline{\sigma},\overline{\delta}]/
(\mathcal{D}/\mathcal{I}\mathcal{D})[t;\overline{\sigma},\overline{\delta}]\overline{f},\quad g\mapsto \overline{g}
$$
has kernel $\mathcal{I}\Lambda$ and
induces an isomorphism of unital nonassociative algebras
\begin{equation}\label{eq:relevantII}
\Lambda/\mathcal{I}\Lambda\cong
(\mathcal{D}/\mathcal{I}\mathcal{D})[t;\overline{\sigma},\overline{\delta}]/(\mathcal{D}/\mathcal{I}\mathcal{D})
[t;\overline{\sigma},\overline{\delta}]\overline{f}, \quad g+\mathcal{I}\Lambda \mapsto \overline{g}
\end{equation}
 over
$$\overline{S_0}={\rm Fix}(\overline{\sigma})\cap {\rm Const}(\overline{\delta})\cap\overline{F} $$
with $\overline{F}=\mathcal{O}_F/\mathcal{I}\mathcal{O}_F$.

\subsection{Example} \label{subsec:iterated}

Let $F$, $L$ and $K$ be number fields and let $K$ be a cyclic extension of both $F$ and $L$ such that
\begin{enumerate}
\item ${\rm Gal}(K/F) = \langle \rho \rangle$ and $[K:F] = n$,
\item ${\rm Gal}(K/L) = \langle \sigma \rangle$ and $[K:L] = m$,
\item $\rho$ and $\sigma$ commute
\end{enumerate}
as in Example \ref{ex:gencyclic}.
Let $F_0=F\cap L$. Let $\mathcal{D}=(\mathcal{O}_K/\mathcal{O}_F, \rho, c)$, $c \in \mathcal{O}_{F_0}$,
be an associative cyclic algebra over $\mathcal{O}_F$ of degree $n$ such that
 $D=(K/F, \rho, c)=\mathcal{D} \otimes_{\mathcal{O}_F}F$ is a division algebra over $F$.
 For $x= x_0 + ex_1 + e^2x_2 +\dots + e^{n-1}x_{n-1}\in D$ where $1,e,\dots,e^{n-1}$ is the standard basis of $D$, define
 $\sigma\in {\rm Aut}_L(D)$  via
$$\sigma(x)=\sigma(x_0) +  \sigma(x_1)e + \sigma(x_2)e^2 +\dots + \sigma(x_{n-1})e^{n-1}.$$
Since $c \in \mathcal{O}_{F_0}$, $\sigma\in {\rm Aut}_L(D)$ has order $m$
and restricts to $\sigma\in {\rm Aut}_{\mathcal{O}_{L}}(\mathcal{D})$.
Let $f=t^m-d\in \mathcal{D}[t;\sigma]$ be irreducible in $D[t;\sigma]$.
Then the $F_0$-algebra $S_f=(D,\sigma,d)$ is a division algebra.
 For coding purposes, usually $d\in \mathcal{O}_F^\times$ or $d\in \mathcal{O}_L^\times$.

A natural order of $(D,\sigma,d)$ is given by the algebra $\Lambda=\mathcal{D}[t;\sigma]/\mathcal{D}[t;\sigma]f$,
and
$$\Lambda=\mathcal{O}_K \oplus  \mathcal{O}_K e\oplus  \dots\oplus\mathcal{O}_K  e^{n-1}t^{m-1}$$
written as a left $\mathcal{O}_K$-module.
Let $\mathcal{I}$ be an ideal in $\mathcal{O}_{F_0}$. Then there is an algebra isomorphism
$$\Lambda/\mathcal{I}\Lambda\cong
(\mathcal{D}/\mathcal{I}\mathcal{D})[t;\overline{\sigma}]/(\mathcal{D}/\mathcal{I}\mathcal{D})
[t;\overline{\sigma}]\overline{f},\quad g+\mathcal{I}\Lambda \mapsto \overline{g}.$$
These are algebras over
 $\overline{F_0}=\mathcal{O}_{F_0}/\mathcal{I}$. This means that the quotient
$\Lambda/\mathcal{I}\Lambda$ is isomorphic to the  $\overline{F_0}$-algebra
$(\overline{D}, \overline{\sigma}, \overline{d}),$
where
$\overline{D}=\mathcal{D}/\mathcal{I}\mathcal{D}$
is a generalized associative cyclic algebra over ${\rm Fix}(\overline{\rho})$.

Note that we will restrict our considerations  to $d\in \mathcal{O}_F^\times$ or $d\in \mathcal{O}_L^\times$
as we are dealing with applications to space-time block codes from now on.

 The right multiplication in $\Lambda$ with a non-zero $x\in \Lambda$ is given by
the $mn \times mn$ matrix $M(x)$ with entries in $\mathcal{O}_K$ obtained by taking the right regular representation $\gamma(x)$ in $\mathcal{D}$ of each
entry  in the $m\times m$-matrix
 \[Y=\gamma(x) = \left[ \begin{array}{ccccc}
x_0 & d \sigma(x_{m-1})& d \sigma^{2}(x_{m-2}) & \cdots & d \sigma^{m-1}(x_1) \\
x_1 & \sigma(x_0) & d \sigma^{2}(x_{m-1}) & \cdots & d \sigma^{m-1}(x_{2}) \\
x_2 & \sigma(x_1) & \sigma^{2}(x_0) & \cdots & d \sigma^{m-1}(x_3)\\
\vdots & \vdots & \vdots & \ddots & \vdots \\
x_{m-1} & \sigma(x_{m-2}) & \sigma^{2}(x_{m-3}) & \cdots & \sigma^{m-1}(x_0) \end{array} \right] \]
which has entries in  $\mathcal{D}$.
Thus $M(x)$  is given by
\begin{equation} \label{equ:matrix_rep_cdaII}
M(x) = \left[ \begin{array}{cccc}
\gamma(x_0) & \gamma(d) \sigma(\gamma(x_{m-1}))&  \cdots & \gamma(d) \sigma^{m-1}(\gamma(x_1)) \\
\gamma(x_1) & \sigma(\gamma(x_0)) & \cdots & \gamma(d) \sigma^{m-1}(\gamma(x_{2})) \\
                                       \vdots & \vdots  & \ddots & \vdots \\
\gamma(x_{m-1}) & \sigma(\gamma(x_{m-2})) & \cdots & \sigma^{m-1}(\gamma(x_0)) \end{array} \right]
\end{equation}
where $\sigma(\gamma(x))$ means we apply $\sigma$ to each entry of the $m\times m$-matrix $\gamma(x)$.
  The matrices $M(x)$ induce a fully diverse  linear space-time
 block code. If $d\in \mathcal{O}_F$, then in particular
$\det(M(x))\in \mathcal{O}_F$ (\cite{MO13}, \cite[Remark 5]{PS15}).
And if $d\in \mathcal{O}_L$, then
\begin{equation} \label{equ:matrix_rep_A}
 M(x) = \begin{bmatrix}
             \gamma(x_0) & d \sigma(\gamma(x_{n-1}))& d \sigma^{2}(\gamma(x_{n-2})) & \cdots & d \sigma^{m-1}(\gamma(x_1)) \\
             \gamma(x_1) & \sigma(\gamma(x_0)) & d \sigma^{2}(\gamma(x_{n-1})) & \cdots & d \sigma^{m-1}(\gamma(x_{2})) \\
             \gamma(x_2) & \sigma(\gamma(x_1)) & \sigma^{2}(\gamma(x_0)) & \cdots & d \sigma^{m-1}(\gamma(x_3))\\
             \vdots & \vdots & \vdots & \ddots & \vdots \\
             \gamma(x_{n-1}) & \sigma(\gamma(x_{n-2})) & \sigma^{2}(\gamma(x_{n-3})) & \cdots & \sigma^{m-1}(\gamma(x_0))
             \end{bmatrix}
\end{equation}
with $d\sigma(\gamma(x_{n-1}))$ etc. denoting the scalar multiplication of the matrix  with $d$ and
$\det(\gamma(M(x))) \in L\cap \mathcal{O}_K=\mathcal{O}_L$ (\cite{R13}, \cite[Lemma 19]{PS15}).

The algebras $A=(D,\sigma,d)$  are behind the fast-decodable iterated codes in \cite{P13.2}, \cite{PS15}, \cite{MO13} \cite{R13}.

Equation (\ref{equ:quotientII}) and the isomorphism in (\ref{eq:relevantII}) imply
 that the right multiplication in $\Lambda/\mathcal{I}\Lambda$ is given by the
$mn \times mn$ matrix in (\ref{equ:matrix_rep_cdaII}) where the entries are read modulo $\mathcal{I}\mathcal{O}_K$.
We call this matrix $M(\overline{x})$.  Therefore we obtain coset codes  by taking the
pre-image of codewords $(M(\overline{x_1}),\dots,M(\overline{x_L}))$  under
$$ \pi:\bigoplus_{i=0}^L {\rm Mat}_{nm}(\mathcal{O}_K) \longrightarrow \bigoplus_{i=0}^L
{\rm Mat}_{nm}(\mathcal{O}_K/\mathcal{I}\mathcal{O}_K).$$

%
%

\section{The structure of quotients of natural orders} \label{sec:structure}

From now on we will only consider the
 generalized nonassociative cyclic algebras  introduced in Sections \ref{ex:nca1} and \ref{subsec:iterated},
as these play an important role in coding theory, both for $\sigma$-constant cyclic linear codes and for
space-time block coding. The setup from \cite{OS} is obtained as a special case whenever $f$ is invariant.

If desired, all the
results  can be generalized verbatim or with slight adjustments to the general cases considered up to now.

 We  look at our two setups separately:

\subsection{Quotients of orders in a nonassociative cyclic algebra}

Let $D=(K/F,\sigma, d)$ be a nonassociative cyclic division algebra
of degree $m$ with $d\in \mathcal{O}_K^\times$ and $\Lambda$ a natural order of $D$.

Unlike in the associative setting we cannot simply limit ourselves to the study of non-zero two-sided ideals $\mathcal{J}$ of $\Lambda$
using a correspondence between them and non-zero ideals of $\mathcal{O}_{F}$, since we do not know whether
$\mathcal{J}\cap \mathcal{O}_{F}$ is a non-zero ideal, instead we only have that
 for every non-zero two-sided ideal $\mathcal{J}$  in $\Lambda$, $\mathcal{I}=\mathcal{J}\cap \mathcal{O}_{K}$
is a non-zero ideal of $\mathcal{O}_{K}$:

\begin{lemma} \label{le:equivI}
(i) Every $a\in \Lambda$ is the zero of the characteristic polynomial of $\gamma(a)$,
 which is a polynomial over $\mathcal{O}_K$. In particular, if $D$ is associative, the elements of
$\Lambda$  are  integral over  $\mathcal{O}_{F}$.
\\ (ii) Let $\mathcal{J}$ be a non-zero two-sided ideal in $\Lambda$. Then
$\mathcal{J}\cap \mathcal{O}_{K}\not=0.$
\end{lemma}

\begin{proof}
(i) Since $K\subset {\rm Nuc}_r(D)$,  left multiplication $L_a$ with any $a\in D^\times$ is a
linear endomorphism of the right $K$-module $D$, so that $L_a\in {\rm End}_{K}(D)$ and thus
$\gamma: D \longrightarrow {\rm End}_{K}(D)\longrightarrow
{\rm Mat}_m(K), a\mapsto L_a\mapsto \gamma(a)$ is a $K$-linear embedding of $K$-vector spaces,
where $\gamma(a)$ is the matrix representing right multiplication in $D$ defined in Example \ref{ex:nca1}.

For $a\in\Lambda$, the entries of the matrix  $\gamma(a)$ are all in $\mathcal{O}_K$ and the characteristic polynomial
of the matrix of $\gamma(a)$ is a polynomial
over $\mathcal{O}_K$. By the Theorem of Cayley-Hamilton, the matrix $\gamma(a)$ inserted into its own characteristic
 polynomial gives the zero matrix. Since
the embedding of $ D$ into ${\rm Mat}_m(K)$ is $K$-linear, or respectively, the embedding of $\Lambda$ into
${\rm Mat}_m(\mathcal{O}_K)$
is $\mathcal{O}_K$-linear, thus $a$ also is a zero of the characteristic polynomial of $\gamma(a)$, a polynomial over
$\mathcal{O}_K$. Therefore any $a\in \Lambda$ is the zero of the characteristic polynomial of $\gamma(a)$.
The second assertion is \cite[Lemma 2]{OS}.
\\ (ii) The proof is similar  to the one of \cite[Lemma 3]{OS}: Let $j\in\mathcal{J}$, $j\not=0$, then $j$
is the zero of the characteristic polynomial of $\gamma(a)$  by (i), which is a polynomial over  $\mathcal{O}_{K}$.
Hence there are $b_i\in \mathcal{O}_{K}$ such that $j^s+b_{s-1}j^{s-1}+\dots+b_1j+s_0=0$.
Suppose that $b_0=b_1=\dots=b_{i-1}=0$ and $b_i=0$, then $j^s+b_{s-1}j^{s-1}+\dots+b_1j+s_0=
j^i(j^{s-i}+\dots+b_{s-1}j^{s-i-1}+b_i)=0$.
Since $D$ is division, $\Lambda$ has no non-trivial zero divisors. Thus $j^{s-i}+\dots+b_{s-1}j^{s-i-1}+b_i=0$
and $b_i\not=0$, implying that $b_i=-(j^{s-i}+\dots+b_{s-1}j^{s-i-1})\in \mathcal{J}\cap \mathcal{O}_K$.
\end{proof}

\begin{remark}
 Note that, contrary to the situation for associative division algebras studied in \cite{OS}, the embedding
$D\longrightarrow {\rm Mat}_m(K)$ only
embeds $D$ into ${\rm Mat}_m(K)$ as a $K$-vector space.
\end{remark}

What we can still say is that  any non-zero ideal $\mathcal{I}$ of $\mathcal{O}_{F}$ lies in the center of
 $\Lambda$ and generates the two-sided ideal
 $\mathcal{I}\Lambda$ where
$$\mathcal{I}\Lambda\cap \mathcal{O}_F=\mathcal{I}.$$
From now on let $\mathcal{I}$ be a non-zero two-sided ideal of $\mathcal{O}_F$, i.e
$\mathcal{I}=\mathfrak{q}_1^{s_1}\cdots \mathfrak{q}_t^{s_t}$ for suitable prime
ideals of $\mathcal{O}_F$, and
$$\mathcal{O}_F/\mathcal{I}=\mathcal{O}_F/\mathfrak{q}_1^{s_1}\cdots\mathfrak{q}_t^{s_t}
\cong \mathcal{O}_F/\mathfrak{q}_1^{s_1}\times \cdots\times \mathcal{O}_F/\mathfrak{q}_t^{s_t},$$
\begin{equation}\label{eq:overring}
\mathcal{O}_K/\mathcal{I}\mathcal{O}_K= \mathcal{O}_K/ \mathfrak{q}_1^{s_1} \cdots\mathfrak{q}_t^{s_t} \mathcal{O}_K\cong
\mathcal{O}_K/ \mathfrak{q}_t^{s_t}\mathcal{O}_K\times\dots\times \mathcal{O}_K/ \mathfrak{q}_t^{s_t} \mathcal{O}_K.
\end{equation}

We immediately obtain from Equation (\ref{eq:generalcyclic}):

\begin{theorem} \label{thm:main}
For
$$\overline{\sigma}(u+\mathcal{I}\mathcal{O}_K)=\sigma(u)+\mathcal{I}\mathcal{O}_K$$
and $\bar{d}=d+\mathcal{I}\mathcal{O}_K$,
\begin{equation}\label{equ:ex5}
\Lambda/\mathcal{I}\Lambda\cong
((\mathcal{O}_K/\mathcal{I}\mathcal{O}_K)/(\mathcal{O}_F/\mathcal{I}),\overline{\sigma},\bar{d})
\end{equation}
is a  generalized nonassociative cyclic algebra over $\mathcal{O}_F/\mathcal{I}$.
\end{theorem}

The map $\Psi$  defined in Section \ref{subsec:naturalI} together with the isomorphism from Equation (\ref{eq:overring})
 implies that
the nonassociative algebra $\Lambda/\mathcal{I}\Lambda$ decomposes into a
product of generalized nonassociative cyclic algebras $D_i=(S_i/R_i,\sigma,d)$ where all the rings $R_i$ and $S_i$ are finite:

\begin{lemma}\label{le:decomp}
The generalized nonassociative cyclic algebra $\Lambda/\mathcal{I}\Lambda$ of Theorem \ref{thm:main} can be described as
a direct sum of generalized nonassociative cyclic algebras, i.e.
\begin{equation}
\Lambda/\mathcal{I}\Lambda\cong
((\mathcal{O}_K/\mathfrak{q}_1^{s_1}\mathcal{O}_K)/(\mathcal{O}_F/\mathfrak{q}_1^{s_1}),\overline{\sigma},d+
\mathfrak{q}_1^{s_1})
\times \dots\times
((\mathcal{O}_K/\mathfrak{q}_t^{s_t}\mathcal{O}_K)/(\mathcal{O}_F/\mathfrak{q}_t^{s_t}),\overline{\sigma},d+\mathfrak{q}_t^{s_t})
\end{equation}
where the respective maps $\overline{\sigma}$ are defined via
$$\overline{\sigma}(u+\mathfrak{q}_j^{s_j}\mathcal{O}_K)=\sigma(u)+\mathfrak{q}_j^{s_j}\mathcal{O}_K.$$
\end{lemma}

This canonically generalizes \cite[Lemma 4]{OS} to the nonassociative setting. The proof is analogous only that
here we are working with nonassociative rings
and thus  homomorphisms between nonassociative rings.
Any non-zero two-sided ideal in $\Lambda/\mathcal{I}\Lambda$ (i.e. of the form $\mathcal{J}/\mathcal{I}\Lambda$
with $\mathcal{I}\Lambda\subset \mathcal{J}$)
  corresponds to a non-zero two-sided ideal in the algebra on the right-hand side.
By classical ideal theory, every ideal of such a product of nonassociative algebras is of the form
$I_1\times\dots\times I_t$ with $I_j$ an ideal of $\mathcal{R}_j$.
W.l.o.g., it therefore suffices to look at the ideals in a generalized nonassociative cyclic algebra
$$((\mathcal{O}_K/\mathfrak{q}^{s}\mathcal{O}_K)/(\mathcal{O}_F/\mathfrak{q}^{s}),
\overline{\sigma},d+\mathfrak{q}^{s})$$
(note that $d+\mathfrak{q}^{s}=0$ is a possibility).
 Such an algebra is a nonassociative finite ring
with $m|\mathcal{O}_K/\mathfrak{q}^{s}\mathcal{O}_K|$ elements and $\mathcal{O}_F/\mathfrak{q}^{s}$ contained in its center.
This observation canonically generalizes \cite[Lemma 4]{OS} and its proof.

Let $m=efg$ with $g$ the number of primes in the factorization of $\mathfrak{q}\mathcal{O}_K$, $e$ the ramification index
and $f$ the inertial degree.  In this paper
we will only look at the unramified case, where $e=1$.

\begin{remark} \label{rem:1}
If $\bar{x}\in \Lambda/\mathcal{I}\Lambda$ is the image of $x\in \Lambda$ then we can take the matrix representing right
multiplication with $x$ and mod the entries by $\mathcal{I}\mathcal{O}_K$ and if we use the algebra for coset coding,
 a codeword
in our coset code $(\gamma(x_1),\dots,\gamma(x_L))$ is a preimage of $(\gamma(\bar{x}_1),\dots,\gamma(\bar{x}_L))$ under
$$ \pi:\bigoplus_{i=0}^{L}{\rm Mat}_m(\mathcal{O}_K) \longrightarrow \bigoplus_{i=0}^{L}{\rm Mat}_m(\mathcal{O}_K/\mathcal{I}\mathcal{O}_K),$$
respectively,  a codeword
in a linear coset code $(x_1,\dots,x_L)$ is a preimage of $(\bar{x}_1,\dots,\bar{x}_L)$ under
$$\pi:\bigoplus_{i=0}^{L}\Lambda\longrightarrow \bigoplus_{i=0}^{L}\Lambda/\mathcal{I}\Lambda.$$
\cite[Remark 3]{OS} holds analogous for our nonassociative setting, i.e. any ideal $I$ in
 $\mathcal{O}_K/\mathfrak{q}^{s}\mathcal{O}_K$ such that $\overline{\sigma}(I)=I$ yields an ideal
 $$\bigoplus_{i=0}^{m-1} It^i$$
 in the nonassociative cyclic algebra
$$((\mathcal{O}_K/\mathfrak{q}^{s}\mathcal{O}_K)/(\mathcal{O}_F/\mathfrak{q}^{s}),
\overline{\sigma},c+\mathfrak{q}^{s}).$$
\end{remark}

\subsection{Quotients of orders in algebras used for iterated space-time block codes}

Let  $D=(K/F,\rho, c)$  with $c\in \mathcal{O}_{F_0}^\times$, $A=(D, \sigma, d)$  be a division algebra as
in Section \ref{subsec:iterated}, $d\in \mathcal{O}_L^\times$ or $d\in \mathcal{O}_F^\times$,
and $\Lambda$  a natural order in $A$.

Let $\mathcal{I}$ be a non-zero two-sided ideal of $\mathcal{O}_{F_0}$, then
 $\mathcal{I}$ lies in the center of  $A=(D, \sigma, d)$ and generates the two-sided ideal
 $\mathcal{I}\Lambda$ of $\Lambda$ with
$$\mathcal{I}\Lambda\cap \mathcal{O}_{F_0}=\mathcal{I}.$$
Write
 $\mathcal{I}=\mathfrak{q}_1^{s_1}\cdots \mathfrak{q}_t^{s_t}$ for suitable prime
ideals of $\mathcal{O}_{F_0}$, and observe that then
$$\mathcal{O}_{F_0}/\mathcal{I}=\mathcal{O}_{F_0}/\mathfrak{q}_1^{s_1}\cdots\mathfrak{q}_t^{s_t}
\cong \mathcal{O}_{F_0}/\mathfrak{q}_1^{s_1}\times \cdots\times \mathcal{O}_{F_0}/\mathfrak{q}_t^{s_t},$$
\begin{equation}
\mathcal{O}_K/\mathcal{I}\mathcal{O}_K= \mathcal{O}_K/ \mathfrak{q}_1^{s_1} \cdots\mathfrak{q}_t^{s_t} \mathcal{O}_K\cong
\mathcal{O}_K/ \mathfrak{q}_t^{s_t}\mathcal{O}_K\times\dots\times \mathcal{O}_K/ \mathfrak{q}_t^{s_t} \mathcal{O}_K.
\end{equation}

\begin{theorem} \label{thm:main2}
Let
$$\overline{\sigma}(u+\mathcal{I}\mathcal{O}_K)=\sigma(u)+\mathcal{I}\mathcal{O}_K,\quad
\overline{\rho}(u+\mathcal{I}\mathcal{O}_K)=\rho(u)+\mathcal{I}\mathcal{O}_K$$
and $\bar{c}=c+\mathcal{I}\mathcal{O}_K$, $d\in \mathcal{O}_L^\times$, $\bar{d}=d+\mathcal{I}\mathcal{O}_K$.
Then
\begin{equation}\label{equ:ex5II}
\Lambda/\mathcal{I}\Lambda\cong (\mathcal{D}/\mathcal{I}\mathcal{D})[t;\overline{\sigma}]/
(\mathcal{D}/\mathcal{I}\mathcal{D})
[t;\overline{{\sigma}}]\overline{f},
\end{equation}
i.e. the right-hand side is the generalized nonassociative  cyclic algebra
$$(\mathcal{D}/\mathcal{I}\mathcal{D},\overline{\sigma} ,\bar{d})$$
over $\overline{F}_0=\mathcal{O}_{F_0}/\mathfrak{p}\mathcal{O}_{F_0}$
 with the algebra
$$\mathcal{D}/\mathcal{I}\mathcal{D} \cong
((\mathcal{O}_K/\mathcal{I}\mathcal{O}_K)/(\mathcal{O}_F/\mathcal{I}),\overline{\rho},\bar{c})$$
 decomposing into a product of generalized associative cyclic algebras by Lemma \ref{le:decomp}.
\end{theorem}

\begin{proof} Equation (\ref{eq:relevantII}) yields the isomorphism
$\Lambda/\mathcal{I}\Lambda\cong
(\mathcal{D}/\mathcal{I}\mathcal{D})[t;\overline{\sigma},\overline{\delta}]/(\mathcal{D}/\mathcal{I}\mathcal{D})
[t;\overline{\sigma},\overline{\delta}]\overline{f},$ $ g+\mathcal{I}\Lambda \mapsto \overline{g}.$
Moreover,
$\mathcal{D}/\mathcal{I}\mathcal{D}$ decomposes into a product of generalized associative cyclic algebras as described in
 Lemma \ref{le:decomp}.
\end{proof}

Let
$\overline{f}(t)=t^n-\overline{d}\in \mathcal{D}[t;\overline{\sigma}]$, $d\in \mathcal{O}_L^\times$
or $d\in \mathcal{O}_F^\times$ and
$$D_j=((\mathcal{O}_K/\mathfrak{q}_j^{s_j}\mathcal{O}_K)/(\mathcal{O}_F/\mathfrak{q}_j^{s_j}),\overline{\rho},
c+\mathfrak{q}_j^{s_j})$$ for $1\leq j\leq t$,
where the respective maps $\overline{\sigma}$, $\overline{\rho}$ are canonically defined via
$$\overline{\sigma}(u+\mathfrak{q}_j^{s_j}\mathcal{O}_K)=\sigma(u)+\mathfrak{q}_j^{s_j}\mathcal{O}_K,$$
$$\overline{\rho}(u+\mathfrak{q}_j^{s_j}\mathcal{O}_K)=\rho(u)+\mathfrak{q}_j^{s_j}\mathcal{O}_K$$
(note that here and later we omit the index $j$ and just write $\overline{\sigma}$, $\overline{\rho}$ for better readability).
We now get the following isomorphism of algebras:

\begin{theorem}\label{prop:righthandside}
In the situation of Theorem \ref{thm:main2},  if
\[
\mathcal{D}/\mathcal{I}\mathcal{D}\cong D_1\times \dots\times D_l
\]
is a product of generalized associative cyclic algebras, then the nonassociative ring $\Lambda/\mathcal{I}\Lambda$ can be described as
\[
\Lambda/\mathcal{I}\Lambda\cong
D_1[t;\overline{\sigma}]/D_1[t;\overline{\sigma}](t^n-\overline{d})
\times \dots\times
D_l[t;\overline{\sigma}]/D_l[t;\overline{\sigma}](t^n-\overline{d})
\]
or alternatively, as
\begin{equation}\label{equ:general}
\Lambda/\mathcal{I}\Lambda\cong (D_1,\overline{\sigma},d+\mathfrak{q}_1^{s_1})
\times \dots\times (D_l,\overline{\sigma},d+\mathfrak{q}_l^{s_t}).
\end{equation}
\end{theorem}

Hence the nonassociative ring $\Lambda/\mathcal{I}\Lambda$ decomposes into a product of
generalized nonassociative cyclic algebras
$$\mathcal{R}_j=(D_j,\overline{\sigma},d+\mathfrak{q}_j^{s_j}).$$
These are finite nonassociative rings with
$$nm|\mathcal{O}_K/\mathfrak{q}_j^{s_j}\mathcal{O}_K|$$
elements.

Any non-zero two-sided ideal in $\Lambda/\mathcal{I}\Lambda$ (i.e. of the form $\mathcal{J}/\mathcal{I}\Lambda$
with $\mathcal{I}\Lambda\subset \mathcal{J}$)
  corresponds to a non-zero two-sided ideal in the algebra on the right-hand side of (\ref{equ:general}) in Theorem \ref{prop:righthandside}.

Every ideal of such a product of nonassociative algebras is of the form
$I_1\times\dots\times I_t$ with $I_j$ an ideal of $\mathcal{R}_j$.
W.l.o.g., it thus suffices to look at the ideals in a generalized nonassociative cyclic algebra
$$ (D_s,\overline{\sigma},d+\mathfrak{q}^{s})=D_s[t;\overline{\sigma}]/D_s[t;\overline{\sigma}]
(t^m-\overline{d}).$$
with $f=t^m-\overline{d}=t^m-d+\mathfrak{q}^{s} \in (\mathcal{O}_K/\mathfrak{q}^{s}\mathcal{O}_K)[t;\overline{\sigma}]$,
$m$ the order of $\overline{\sigma}$.
This observation canonically generalizes \cite[Lemma 4]{OS} to the nonassociative setting.
Let $mn=efg$ with $g$ being the number of primes in the factorization of $\mathfrak{q}\mathcal{O}_K$, $e$ the ramification index
and $f$ the inertial degree.

We again only look at the unramified case, where $e=1$.

\begin{remark}
Analogously as described in Remark \ref{rem:1},
 if $\bar{x}\in \Lambda/\mathcal{I}\Lambda$ is the image of $x\in \Lambda$ then we can take the matrix representing right
multiplication with $x$ and mod the entries by $\mathcal{I}\mathcal{O}_K$ and if we use the algebra for coset coding, a
codeword in our space-time block coset code $(M(x_1),\dots,M(x_L))$ is a preimage of $(M(\bar{x}_1),\dots,M(\bar{x}_L))$ under
$$ \pi:\bigoplus_{i=0}^L {\rm Mat}_{nm}(\mathcal{O}_K) \longrightarrow \bigoplus_{i=0}^L {\rm Mat}_{nm}(\mathcal{O}_K/\mathcal{I}\mathcal{O}_K),$$
respectively,  a codeword
in our linear coset code $(x_1,\dots,x_L)$ is a preimage of $(\bar{x}_1,\dots,\bar{x}_L)$ under
$$ \pi:\bigoplus_{i=0}^L\Lambda\longrightarrow \bigoplus_{i=0}^L \Lambda/\mathcal{I}\Lambda.$$
Any ideal $I$ in
 $\mathcal{O}_K/\mathfrak{q}^{s}\mathcal{O}_K$ such that $\overline{\sigma}(I)=I$ and $\overline{\rho}(I)=I$
 yields an ideal
 $$I\oplus Ie \oplus \dots\oplus  Ie^{n-1}\oplus \dots\oplus Ie^{n-1}t^{m-1}$$
(we denote the canonical basis of any $D_i$ by $1,e,\dots,e^{n-1}$ for ease of notation) in the algebra
 $$(D_i,\overline{\sigma},d+\mathfrak{q}^{s})=D_i[t;\overline{\sigma}]/D_i[t;\overline{\sigma}]
(t^m-\overline{d}).$$
\end{remark}

We note that \cite[(11)]{OS} also holds for the codes obtained in our nonassociative setting: For $X_i=M(x_i)$ and $\mathcal{J}=(\alpha)$,
$\alpha\in \mathcal{O}_{F_0}$, we obtain a lower bound for the minimum determinant
$\Delta_{min}$ of $\overline{\mathcal{C}}$:
\begin{equation} \label{equ:key}
\Delta_{min}\geq {\rm min}_{0\not=X_i}|\det(X_i)|^2 {\rm min}(d_H(\overline{\mathcal{C}})^2,|\alpha|^{2n} ),
\end{equation}
where $d_H(\overline{\mathcal{C}})$ is the Hamming distance of $\overline{\mathcal{C}}$.

\subsection{Codes with prescribed minimum distance}
The construction mentioned in \cite[Example 5]{OS} works here as well and designs
a code $\mathcal{C}$ such that $\overline{\mathcal{C}}$ has prescribed minimum distance:
suppose $\overline{K}=\mathcal{O}_K/\mathcal{I}\mathcal{O}_K$ is a field.
Choose a code $\mathcal{B}$ of desired minimum distance $d_H$ and length $L$ over $\overline{K}$ and incorporate
the entries of the
code into the first summand of the right hand side of a generalized nonassociative cyclic algebra as in (\ref{equ:ex5}):
$$\Lambda/\mathcal{I}\Lambda\cong
((\mathcal{O}_K/\mathcal{I}\mathcal{O}_K)/(\mathcal{O}_F/\mathcal{I}),\overline{\sigma},\bar{d}))
=\bigoplus_{i=0}^{m-1}(\mathcal{O}_K/\mathcal{I}\mathcal{O}_K)t^i,
$$
or into the first summand of the right hand side of a generalized nonassociative cyclic algebra as in (\ref{equ:ex5II}):
$$\Lambda/\mathcal{I}\Lambda\cong (\mathcal{D}/\mathcal{I}\mathcal{D},\overline{\sigma} ,\bar{d})
\cong \bigoplus_{i,j=0}^{i=m-1,j=n-1}(\mathcal{O}_K/\mathcal{I}\mathcal{O}_K)e^it^j$$
(with $1,e,\dots,e^{m-1}$ denoting the canonical basis of $\mathcal{D}/\mathcal{I}\mathcal{D}$).

The matrices representing the right multiplication in
$\Lambda/\mathcal{I}\Lambda$
form a subset of ${\rm Mat}_{m}(\overline{K})$
(respectively, of ${\rm Mat}_{mn}(\overline{K})$) and are obtained by
taking the  matrices representing the right multiplication in $\Lambda$ and then
 modding out
the entries of each matrix  by $\mathcal{I}\mathcal{O}_K$.
Choose as outer code $\overline{\mathcal{C}}$ the $L$-tuples of matrices coming from $\Lambda/\mathcal{I}\Lambda$ such that
  $(x_{1,0},\dots,x_{L,0})$ belong to $\mathcal{B}$ as explained in \cite[Example 5]{OS}, since the fact that we might
   be dealing also with nonassociative algebras here is not relevant in the argument.

%
%

 \section{The inertial case, where $g=e=1$ and $\mathcal{I}=\mathfrak{q}$}\label{sec:inertial}

 \subsection{Nonassociative cyclic division algebras}
In the terminology of Section \ref{ex:nca1}, let $g=e=1$.
Then $\mathcal{I}=\mathfrak{q}\subset \mathcal{O}_F$ remains a prime ideal in $\mathcal{O}_K$ and
$\mathfrak{q}\mathcal{O}_K= \mathcal{Q}$ for a prime $\mathcal{Q}$ of $\mathcal{O}_K$
with inertial degree $f=m$. The finite field $\overline{K}=\mathcal{O}_K/\mathfrak{q}\mathcal{O}_K=\mathcal{O}_K/ \mathcal{Q}$
is a cyclic Galois field extension  of degree $m$ of  $\overline{F}=\mathcal{O}_F/\mathfrak{q}$ with
${\rm Gal}(\overline{K}/\overline{F})=\langle \overline{\sigma}\rangle$.

 Let $A=(K/F,\sigma,c)$  be a nonassociative cyclic division algebra of degree $m$,
 $c\in \mathcal{O}_K\setminus \mathcal{O}_F$, with natural order $\Lambda$
 and
$$\Lambda/\mathcal{I}\Lambda\cong (\overline{K}/\overline{F},\overline{\sigma},\overline{c})$$
with $\overline{c}=c+\mathfrak{q}$.
 Since $c\not\in  \mathcal{O}_F$ it is clear that  $c\not\in \mathfrak{q}$.

\begin{theorem}\label{thm:inertial}
Let $\mathcal{I}=\mathfrak{q}$ be a prime ideal in $\mathcal{O}_F$ which is inert in $\mathcal{O}_K$, and
$\mathfrak{q}\mathcal{O}_K= \mathcal{Q}$, $\mathcal{Q}$ a prime ideal in $\mathcal{O}_K$.
Then
$$\Lambda/\mathcal{I}\Lambda\cong(\overline{K}/\overline{F},\overline{\sigma},\overline{c})$$
 is a  nonassociative cyclic
 algebra of degree $m$ over the finite field $\overline{F}$.
If $m$ is prime or if $1,\bar{c},\dots,\bar{c}^{m-1}$ are linearly independent over $\overline{F}$, then this is a
central simple division
algebra and hence
 the only proper two-sided ideal $\mathcal{J}$ of $\Lambda$ that contains $\mathcal{I}=\mathfrak{q}$ is
$$\mathcal{I}\Lambda=\bigoplus_{j=0}^{m-1}\mathfrak{q}\mathcal{O}_Kt^j.$$
\end{theorem}

\begin{proof}
Since $c\not\in \mathfrak{q}$, $\bar{c}\in \overline{F}^\times$,
$\Lambda/\mathcal{I}\Lambda\cong(\overline{K}/\overline{F},\overline{\sigma},\overline{c})$ is a nonassociative cyclic
 algebra of degree $m$ over $\overline{F}$. If $m$ is a prime or if $1,\bar{c},\dots,\bar{c}^{m-1}$ are linearly independent over $\overline{F}$,
 it is a division algebra \cite{S12} and therefore
has only  trivial two-sided ideals.
Thus the only proper two-sided ideal $\mathcal{J}$ of $\Lambda$ that contains $\mathcal{I}$ is $\mathcal{I}\Lambda=
\bigoplus_{j=0}^{m-1}\mathfrak{q}\mathcal{O}_Kt^j$.
\end{proof}

\begin{example}
Let $\omega_7$ be a primitive 7th root of unity, $K=\mathbb{Q}(\omega_3,\omega_7+\omega_7^{-1})$, $F=\mathbb{Q}(\omega_3),$ and let
$A=(\mathbb{Q}(\omega_3,\omega_7+\omega_7^{-1})/\mathbb{Q}(\omega_3),\sigma,c)$ with
$c\in \mathcal{O}_K\setminus \mathcal{O}_F$ be a nonassociative cyclic division algebra of degree 3. Then
$$\Lambda=\mathbb{Z}[\omega_3,\omega_7+\omega_7^{-1}]\oplus \mathbb{Z}[\omega_3,\omega_7+\omega_7^{-1}] t\oplus \mathbb{Z}[
\omega_3,\omega_7+\omega_7^{-1}]t^2$$
is the natural order in $A$. Let $\mathcal{I}=\langle 2\rangle$, which is a prime ideal in $\mathcal{O}_F=
\mathbb{Z}[\omega_3]$, then
this ideal remains prime in $\mathcal{O}_K=\mathbb{Z}[\omega_3,\omega_7+\omega_7^{-1}]$ and
$\mathbb{Z}[i]/\mathcal{I}\cong\mathbb{F}_4$.  Since $\mathcal{I}$ is inert in $\mathbb{Q}(\omega_3,\omega_7+\omega_7^{-1})$, we have
 by Theorem \ref{thm:inertial},
$$\Lambda/\mathcal{I}\Lambda\cong(\mathbb{F}_{64}/\mathbb{F}_4,\overline{\sigma},\overline{c})$$
 is a nonassociative cyclic division
 algebra of degree $3$ over $\mathbb{F}_4$ for all $\overline{c}\not=0$. It follows that $\mathcal{J}=\langle 2\rangle\Lambda$
 is  the only proper two-sided ideal
  of $\Lambda$ that contains $\langle 2\rangle\Lambda$ and
  $\Lambda/\mathcal{J}\cong (\mathbb{F}_{64}/\mathbb{F}_4,\overline{\sigma},\overline{c})$.
\end{example}

\begin{example}
Let $\omega_{15}$ be a primitive 15th root of unity, $K=\mathbb{Q}(i,\omega_{15}+\omega_{15}^{-1})$,
$F=\mathbb{Q}(i),$ and let
$D=(\mathbb{Q}(i,\omega_{15}+\omega_{15}^{-1})/\mathbb{Q}(i),\sigma,c)$ with
$c\in \mathcal{O}_K\setminus \mathcal{O}_F$ be a nonassociative cyclic division algebra of degree 4 (i.e., we choose $c$
such that $1,c,c^2,c^3$ are linearly independent). Then
$$\Lambda=\mathbb{Z}[i,\omega_{15}+\omega_{15}^{-1}]\oplus \mathbb{Z}[i,\omega_{15}+\omega_{15}^{-1}] t\oplus
\mathbb{Z}[i,\omega_{15}+\omega_{15}^{-1}]t^2\oplus \mathbb{Z}[i,\omega_{15}+\omega_{15}^{-1}]t^3$$
is the natural order in $D$.

Let
$\langle 2\rangle\subset \mathbb{Z}$, which is the square of the  ideal $\langle 1+i\rangle$ in
$\mathcal{O}_F=\mathbb{Z}[i]$. That is, $\langle 2\rangle$ is totally ramified in $\mathbb{Q}(i)$ and
$\mathbb{Z}[i]/\langle 1+i\rangle\cong\mathbb{F}_2$.
Let $\mathcal{I}=\mathfrak{q}=\langle 1+i\rangle$. Then $\langle 1+i\rangle$ is unramified in
$\mathbb{Q}(i,\omega_{15}+\omega_{15}^{-1})$. By Theorem \ref{thm:inertial}, we have
$$\Lambda/\mathcal{I}\Lambda\cong(\mathbb{F}_{16}/\mathbb{F}_2,\overline{\sigma},\overline{c})$$
 is a nonassociative cyclic
 algebra of degree $4$ over $\mathbb{F}_4$ which is never a division algebra, since $1,c,c^2,c^3$ are always
 linearly dependent over $\mathbb{F}_2$. Hence
$$\Lambda/\mathcal{I}\Lambda\cong(\mathbb{F}_{16}/\mathbb{F}_2,\overline{\sigma},\overline{c})$$
is a nonassociative cyclic algebra with zero divisors and $f=t^4-\overline{c}$ is reducible in $\mathbb{F}_{16}[t;\overline{\sigma}]$.
\end{example}

\begin{remark}
In the nonassociative case,  it is very easy to make sure
the algebra employed is division.
Space-time block codes designed using cyclic division algebras which are not associative are fully diverse, however,
a non-vanishing determinant cannot be achieved in most cases. For scenarios like
the multiple-input double-output code design, it can be worth trading the non-vanishing
determinant for fast-decodability, however \cite{SPO12}.
\end{remark}

\begin{example}\label{ex:inert}
Let $K=\mathbb{Q}(i,\sqrt{5})$, $F=\mathbb{Q}(i)$, $D=(\mathbb{Q}(i,\sqrt{5})/\mathbb{Q}(i),\sigma,c)$ with $c\in \mathcal{O}_K\setminus \mathcal{O}_F$ be a
nonassociative quaternion algebra. The automorphism $\sigma:K
\rightarrow K$ is defined by $\sigma(i) = -i$. Then
$$\Lambda=\mathbb{Z}[i,(1+\sqrt{5})/2]\oplus \mathbb{Z}[i,(1+\sqrt{5})/2] t$$
is the natural order in $D$. Let $\mathcal{I}=\mathfrak{q}=\langle 1+i\rangle\subset\mathbb{Z}[i]$, then
$\mathbb{Z}[i]/\mathcal{I}\cong\mathbb{F}_2$. Since $\mathcal{I}$ is inert in $\mathbb{Q}(i,\sqrt{5})$, we have
$$\Lambda/\mathcal{I}\Lambda\cong(\mathbb{F}_4/\mathbb{F}_2,\overline{\sigma},\overline{c})$$
 is a nonassociative quaternion
 algebra over $\mathbb{F}_2$  for all $\overline{c}\not=0$. It follows that $\mathcal{J}=\langle 1+i\rangle\Lambda$
 is  the only proper two-sided ideal
  of $\Lambda$ that contains $\langle 1+i\rangle\Lambda$ and
  $\Lambda/\mathcal{J}\cong (\mathbb{F}_4/\mathbb{F}_2,\overline{\sigma},\overline{c})$.

As in \cite[Example 4]{OS},
 we can choose the coset code $\mathcal{C}'=\{(\gamma(x_0),\gamma(x_1),\gamma(x_2))\in {\rm Mat}_2(\mathcal{O}_K)\,|\, \gamma(x_2)=
 \gamma(x_0)+\gamma(x_1)\}$ as preimage from
the space-time codeword
$$(\gamma(\overline{x_0}),\gamma(\overline{x_1})\gamma(\overline{x_0})+\gamma(\overline{x_1})).$$
Recall that a coset code constructed from this nonassociative cyclic algebra however would  not have non-vanishing determinant.
\end{example}

 \subsection{Generalized nonassociative cyclic division algebras}

In the terminology of Section \ref{subsec:iterated}, let $g=e=1$.
Then $\mathcal{I}=\mathfrak{q}\subset \mathcal{O}_{F_0}$ remains a prime ideal in $\mathcal{O}_K$ and
$\mathfrak{q}\mathcal{O}_K= \mathcal{Q}$ for a prime $\mathcal{Q}$ of $\mathcal{O}_K$
with inertial degree $f=m$. The finite field $\overline{K}=\mathcal{O}_K/\mathfrak{q}\mathcal{O}_K=\mathcal{O}_K/ \mathcal{Q}$
is a cyclic Galois field extension  of degree $m$ of  $\overline{F}=\mathcal{O}_F/\mathfrak{q}$ with
${\rm Gal}(\overline{K}/\overline{F})=\langle \overline{\rho}\rangle$.
The finite field $\overline{K}=\mathcal{O}_K/\mathfrak{q}\mathcal{O}_K
=\mathcal{O}_K/ \mathcal{Q}$
is a cyclic Galois field extension of the field $\overline{L}=\mathcal{O}_L/\mathfrak{q}$ of degree $n$ with
${\rm Gal}(\overline{K}/\overline{F})=\langle \overline{\sigma}\rangle$.

Let $A=(D, \sigma, d)$ with  $D=(K/F,\rho, c),$
 $c\in \mathcal{O}_{F_0}^\times$, $d\in \mathcal{O}_L^\times$ or $d\in \mathcal{O}_F^\times$,
 be a division algebra as in Section \ref{subsec:iterated}, and $\Lambda$ be a natural order in $A$.

From Proposition \ref{prop:righthandside} we know that
\[
\mathcal{D}/\mathcal{I}\mathcal{D}\cong D_1\times \dots\times D_t
\]
is a product of generalized associative cyclic algebras and
\[
\Lambda/\mathcal{I}\Lambda\cong
D_1[t;\overline{\sigma}]/D_1[t;\overline{\sigma}](t^n-\overline{d})
\times \dots\times
D_l[t;\overline{\sigma}]/D_l[t;\overline{\sigma}](t^n-\overline{d}),
\]
i.e.,
\begin{equation}\label{equ:importantII}
\Lambda/\mathcal{I}\Lambda\cong (D_1,\overline{\sigma},d+\mathfrak{q}_1^{s_1})
\times \dots\times (D_l,\overline{\sigma},d+\mathfrak{q}_l^{s_l}).
\end{equation}

Here, the
$$D_i\cong(\overline{K}/\overline{F},\overline{\rho},\overline{c}), \quad 1\leq i\leq l,$$
 are generalized associative cyclic algebras of degree $n$ over the finite field $\overline{F}$.
By \cite[Proposition 1, Proposition 2]{OS},  if $c\not\in \mathfrak{q}$ then
$$D_i\cong(\overline{K}/\overline{F},\overline{\rho},\overline{c})\cong {\rm Mat}_n(\overline{F}),$$
and  the only proper two-sided ideal $\mathcal{J}_i$ in the natural order of $D_i$  that contains $\mathcal{I}=\mathfrak{q}$ is
$$\bigoplus_{j=0}^{n-1}\mathfrak{q}\mathcal{O}_Ke^j.$$
If  $c\in \mathfrak{q}$ then
$$D_i\cong(\overline{K}/\overline{F},\overline{\rho},0)\cong \overline{K}[t;\overline{\rho}]/(e^n),$$
and the only two-sided ideals $\mathcal{J}_i$
 in the natural order of $D_i$  that contain $\mathcal{I}=\mathfrak{q}$ are the ideals
$(e^j)$, $1\leq j\leq n-1.$
The two-sided ideals of $\Lambda/\mathcal{I}$ thus have the form $\mathcal{J}_1\times\dots \times \mathcal{J}_l$
with the $\mathcal{J}_k$ of the corresponding type.

\begin{example}
Let $\omega=\omega_3$ denote the primitive third root of unity, $\theta = \omega_7 + \omega_7^{-1} = 2
\cos(\frac{2 \pi}{7})$ where $\omega_7$ is a primitive $7^{th}$ root of
unity and let $F = \mathbb{Q}(\theta)$.

 Let $K = F(\omega) =\mathbb{Q}(\omega, \theta)$ and take
 the quaternion division algebra $D = (K/F, \sigma, -1)$, where $\sigma:i \mapsto -i$.
In particular, this means $\sigma(\omega) = \omega^2$. Let $L =\mathbb{Q}(\omega)$ so that $K/L$ is a cubic cyclic field extension
whose Galois group is generated by the automorphism
$$\tau: \omega_7 +\omega_7^{-1} \mapsto \omega_7^2 + \omega_7^{-2}.$$
Note that $\omega\in\mathcal{O}_L=\mathbb{Z}[\omega]$.
The algebra $A = (D, \tau, \omega) $  is used in the codes employed in \cite{R13} (cf. \cite{SP14}).
Since $\omega \neq z \tl(z) \tl^2(z)$ for all $z \in D$, $A$ is division \cite{R13}.

Here
$$\Lambda=\mathbb{Z}[\omega,\omega_7+\omega_7^{-1}]\oplus \mathbb{Z}[\omega,\omega_7+\omega_7^{-1}] e\oplus \mathbb{Z}[
\omega,\omega_7+\omega_7^{-1}]e^2\oplus\dots $$
is the natural order in $A$.

Let $\mathcal{I}=\langle 2\rangle$, which is also a prime ideal in $\mathcal{O}_F=\mathbb{Z}[\omega]$ and
this ideal remains prime in $\mathcal{O}_K=\mathbb{Z}[\omega,\omega_7+\omega_7^{-1}]$ and
$\mathbb{Z}[i]/\mathcal{I}\cong\mathbb{F}_4$.
 $\mathcal{I}$ is inert in $K=\mathbb{Q}(\omega,\omega_7+\omega_7^{-1})$.
 We have that
$$\Lambda/\mathcal{I}\Lambda\cong ((\mathbb{F}_{64}/\mathbb{F}_8,\overline{\sigma},-1),\overline{\tau},
\overline{\omega})$$
 is a generalized nonassociative cyclic  algebra employing the split quaternion
 algebra
 $$(\mathbb{F}_{64}/\mathbb{F}_8,\overline{\sigma},-1)\cong {\rm Mat}_2(\mathbb{F}_8)$$ over $\mathbb{F}_8$
 in its construction,
 where $\overline{\omega}\in \mathbb{Z}[\omega,\omega_7+\omega_7^{-1}]\setminus \mathbb{Z}[\omega_7+\omega_7^{-1}]$.
\end{example}

 \section{Some more case studies} \label{sec:inertialII}

Without striving to cover all  cases, we proceed to give some more examples of the nonassociative rings
we obtain as quotients $\Lambda/\mathcal{I}\Lambda$ of $\Lambda$.

\subsection{$A=(K/F,\sigma,d)$ and $\mathcal{I}=\mathfrak{q}^s$}
 Let $A=(K/F,\sigma,d)$ be a nonassociative cyclic division algebra, $c\in\mathcal{O}_K\setminus \mathcal{O}_F$,
 with the natural order $\Lambda$.
 Let $\mathcal{I}=\mathfrak{q}^s$ be a power of a prime ideal $\mathfrak{q}$ in $\mathcal{O}_F$, $s>1$.
 We assume that $\mathfrak{q}$
 is inert in $\mathcal{O}_K$, i.e. $\mathfrak{q}\subset\mathcal{O}_F$ stays prime in $\mathcal{O}_K$, so
 $g=e=1$, $f=m$. Define
$\mathfrak{q}\mathcal{O}_K= \mathcal{Q}$, with $\mathcal{Q}$ a prime ideal in $\mathcal{O}_K$.
 Then
 $$\mathcal{O}_K/\mathfrak{q}^s\mathcal{O}_K\cong \mathcal{O}_K/\mathcal{Q}^s$$
 and
$$\Lambda/\mathcal{I}\Lambda\cong ((\mathcal{O}_K/\mathcal{Q}^s)/(\mathcal{O}_F/\mathfrak{q}^s),\overline{\sigma},\overline{c})$$
with $\overline{c}=c+\mathcal{Q}^s$,
$$\overline{\sigma}(x+\mathcal{Q}^s)=\sigma(x)+\mathcal{Q}^s,$$
is a generalized nonassociative cyclic algebra over $\mathcal{O}_F/\mathfrak{q}^s$.
 Since $d\in \mathcal{O}_K\setminus \mathcal{O}_F$, we know  that $d\not\in \mathfrak{q}$.

\subsection{$A=(D,\sigma,d)$, $\mathcal{I}=\mathfrak{q}^s$ and $c\not\in \mathfrak{q}$}

 Let $A=(D,\sigma,d)$ as in Section \ref{subsec:iterated}, and assume $d\in \mathcal{O}_L$
 invertible (if $d\in \mathcal{O}_F$ invertible a similar argument applies). Let
$D=(K/F,\sigma,c)$, $c\in\mathcal{O}_{F_0}$, and take a natural order $\Lambda$ in $A$.

Let $\mathcal{I}=\mathfrak{q}^s$ be a power of a prime ideal $\mathfrak{q}$ in $\mathcal{O}_{F_0}$, $s>1$.
 We assume that $\mathfrak{q}$
 is inert in $\mathcal{O}_K$, i.e. $\mathfrak{q}$ stays prime in $\mathcal{O}_K$.
  Define
$\mathfrak{q}\mathcal{O}_K= \mathcal{Q}$, $\overline{F}=\mathcal{O}_F/\mathfrak{q}$, with
$\mathcal{Q}$ a prime ideal in $\mathcal{O}_K$
and
 $$\mathcal{O}_K/\mathfrak{q}^s\mathcal{O}_K\cong \mathcal{O}_K/\mathcal{Q}^s.$$
  Then
$$\Lambda/\mathcal{I}\Lambda\cong (\overline{D},\overline{\sigma},d+\mathfrak{q}^{s})=\overline{D}[t;\overline{\sigma}]/
\overline{D}[t;\overline{\sigma}](t^m-\overline{d})$$
with
$$\overline{D}=\mathcal{D}/\mathcal{I}\mathcal{D}=((\mathcal{O}_K/\mathfrak{q}^{s_j}\mathcal{O}_K)/(\mathcal{O}_F/\mathfrak{q}^{s}),\overline{\rho},
c+\mathfrak{q}_t^{s}),$$
$$f=t^m-d+\mathfrak{q}^{s} \in (\mathcal{O}_L/\mathfrak{q}^{s}\mathcal{O}_L)[t;\overline{\sigma}],$$
$m$ the order of $\overline{\sigma}$.

Suppose that  $c\not\in \mathfrak{q}$, then we know by \cite[Proposition 3]{OS} that
$$\mathcal{D}/\mathcal{I}\mathcal{D}\cong{\rm Mat}_n(\mathcal{O}_F/\mathfrak{q}^{s})$$
splits.

\subsection{$A=(K/F,\sigma,d)$, the split case}

 Let $A=(K/F,\sigma,d)$ be a nonassociative cyclic division algebra, $d\in\mathcal{O}_K\setminus \mathcal{O}_F$,
 with the natural order $\Lambda$.

Suppose that $\mathcal{I}=\mathfrak{q}\subset\mathcal{O}_{F}$ is a prime ideal that factors as
$$\mathfrak{q}\mathcal{O}_K=\mathcal{Q}_1\dots\mathcal{Q}_g$$
for some $g>1$. Define the field
$\overline{F}=\mathcal{O}_F/\mathfrak{q}$ and let $\overline{K}=\mathcal{O}_K/\mathfrak{q}\mathcal{O}_K$,  then the Chinese Remainder Theorem tells us that
$$\overline{K}\cong \overline{K}^{(1)}\times\dots\times\overline{K}^{(g)}$$
and consequently this isomorphism induces an automorphism $\overline{\sigma}$ on
$\overline{K}\cong \overline{K}^{(1)}\times\dots\times\overline{K}^{(g)}$. Analogously as in \cite[(15)]{OS} we thus have
an isomorphism, this time  of generalized nonassociative cyclic algebras over the field $\overline{F}$, given by
$$\Lambda/\mathcal{I}\Lambda\cong ((\overline{K}^{(1)}\times\dots
\times\overline{K}^{(g)})/\overline{F},\overline{\sigma},\overline{d}).$$
As proved in \cite[Lemma 5]{OS}, here
$$\overline{K}^{(i)}/\overline{F}$$ is a cyclic field extension with Galois group generated by $\overline{\sigma}^g$,
and
$$\overline{K}^{(i)}\cong \overline{K}^{(1)} \text{ for all }i\in\{2,\dots,m-1\}.$$
Moreover, after suitably reordering the primes $ \mathcal{Q}_1,\dots,\mathcal{Q}_g$ we can assume the
action of $\overline{\sigma}^j$ on $(k,0,\dots,0)\in \overline{K}^{(1)}\times\dots\times\overline{K}^{(g)}$
is given by $(0,\dots,\overline{\sigma}^j(k),\dots,0)$ with $\overline{\sigma}^j(k)$ being in the slot $j+1$, read modulo $g$.

We have $d\not\in\mathfrak{q}$.

\subsection{$A=(D,\sigma,d)$, the split case}

 Let $A=(D,\sigma,d)$ as in Section \ref{subsec:iterated}, and assume $d\in \mathcal{O}_L$
 invertible (if $d\in \mathcal{O}_F$ invertible a similar argument applies). Let
$D=(K/F,\sigma,c)$, $c\in\mathcal{O}_{F_0}$, and take a natural order $\Lambda$ in $A$.

Suppose that $\mathfrak{q}\subset\mathcal{O}_{F}$ is still a prime ideal, but that $\mathcal{I}=\mathfrak{q}\subset\mathcal{O}_{F_0}$ factors as $\mathfrak{q}\mathcal{O}_K=
\mathcal{Q}_1\dots\mathcal{Q}_g$
for some $g>1$.
Define the fields $\overline{F_0}=\mathcal{O}_{F_0}/\mathfrak{q}$ and
$\overline{F}=\mathcal{O}_F/\mathfrak{q}\mathcal{O}_F$.
The Chinese Remainder Theorem tells us that for $\overline{K}=\mathcal{O}_K/\mathfrak{q}\mathcal{O}_K$, we have
$$\overline{K}\cong \overline{K}^{(1)}\times\dots\times\overline{K}^{(g)}$$
and consequently this induces an automorphism $\overline{\sigma}$ on
$\overline{K}\cong \overline{K}^{(1)}\times\dots\times\overline{K}^{(g)}$. By \cite[(15)]{OS} we thus have
an isomorphism
$$\mathcal{D}/\mathcal{I}\mathcal{D}\cong ((\overline{K}^{(1)}\times\dots
\times\overline{K}^{(g)})/(\mathcal{O}_F/\mathcal{I}),\overline{\rho},\overline{c})$$
 of cyclic algebras over the field $\overline{F}$.
As proved in \cite[Lemma 5]{OS},
 after suitably reordering the primes $ \mathcal{Q}_1,\dots,\mathcal{Q}_g$ we can assume that the
action of $\overline{\rho}^j$ on $(k,0,\dots,0)\in \overline{K}^{(1)}\times\dots\times\overline{K}^{(g)}$
is given by $(0,\dots,\overline{\rho}^j(k),\dots,0)$ with $\overline{\rho}^j(k)$ being in the slot $j+1$, read modulo $g$.
There are two cases to consider:
\\
If $\overline{c}\not\in\mathfrak{q}$ then
$$\mathcal{D}/\mathcal{I}\mathcal{D}\cong {\rm Mat}_n(\overline{F})$$
\cite[Proposition 4]{OS} and hence
$$\Lambda/\mathcal{I}\Lambda\cong ({\rm Mat}_n(\overline{F}),\overline{\sigma},\overline{d}).$$
If $\overline{c}\in\mathfrak{q}$ then
$$\mathcal{D}/\mathcal{I}\mathcal{D}\cong
((\overline{K}^{(1)}\times\dots\times\overline{K}^{(g)})/\overline{F},\overline{\rho},\overline{c} )$$
\cite[Proposition 5]{OS} and thus
$$\Lambda/\mathcal{I}\Lambda\cong (\mathcal{D}/\mathcal{I}\mathcal{D},\overline{\sigma},\overline{d}).$$
Note that both times $\overline{d}\not\in\mathfrak{q}$ since we look at algebras which are not associative.

\section{Conclusion and future work} \label{sec:last}

Our approach canonically generalizes and unifies the ones of \cite{DO.0} and  \cite{DO}: the situation considered there
only deals with natural orders in associative cyclic division algebras, i.e. where $f=t^m-d\in \mathcal{O}_K[t;\sigma]$
is irreducible and
$K/F$ is a cyclic number field extension of degree $m$
with Galois group generated by $\sigma$.

We leave it to coding specialists to find well performing codes over the nonassociative finite rings we have presented here,
and to decide to which coding scenarios they can be best applied. We suspect there are applications to wiretap coding, similarly as
outlined in \cite[Section 8]{OS} using the  way to design the wiretap lattice codes presented in \cite{BO}.

It also suggests that future work might
look at the different possible constructions of
linear codes over finite chain rings which arise from nonassociative algebras obtained from a skew polynomial
ring, as it can be seen as yet another generalization of Construction A of lattices from linear codes which are
defined using the quotient $\Lambda/p$ for some suitable prime ideal $p$.


\end{document}